\documentclass{amsart}
\usepackage{amsfonts}
\usepackage{amssymb} 
\usepackage{graphicx}
\usepackage{enumerate}
\usepackage{amsmath}
\usepackage{comment}

\newtheorem{theorem}{Theorem}[section]
\newtheorem{lemma}[theorem]{Lemma}

\theoremstyle{definition}

\theoremstyle{remark}

\numberwithin{equation}{section}

\newcommand{\refth}[1]{Theorem~\ref{#1}}

\newcommand{\refeq}[1]{(\ref{#1})}

\newcommand{\beeq}[1]{\begin{equation} \label{#1}}
\newcommand{\eeq}{\end{equation}}
\renewcommand{\(}{\begin{eqnarray*}}
\renewcommand{\)}{\end{eqnarray*}}
\newcommand{\beeqn}{\begin{eqnarray}}
\newcommand{\eeqn}{\end{eqnarray}}

\newcommand{\eqwhere}{\mbox{\hspace*{3mm} where \hspace*{3mm}}}

\newcommand{\ltt}[1]{{\large \tt #1}}

\renewcommand{\quad}{\hspace*{5mm}}
\renewcommand{\qquad}{\hspace*{10mm}}

\newcommand{\lp}{\left(  }
\newcommand{\rp}{\right) }
\newcommand{\lb}{\left\{  }
\newcommand{\rb}{\right\} }
\newcommand{\lbr}{\left[  }
\newcommand{\rbr}{\right] }

\newcommand{\lc}{\left\lceil}
\newcommand{\rc}{\right\rceil}

\newcommand{\Z}{{\mathbb Z}}

\newcommand{\R}{{\mathbb R}}

\newcommand{\ep}{\epsilon}
\newcommand{\FF}{{\mathcal F}}
\newcommand{\GG}{{\mathcal G}}
\newcommand{\DD}{{\mathcal D}}

\newcommand{\TT}{{\mathcal T}}
\renewcommand{\AA}{{\mathcal A}}
\newcommand{\EE}{{\mathcal E}}

\newcommand{\MM}{{\mathcal M}}
\newcommand{\PP}{{\mathcal P}}

\newcommand{\CC}{{\mathcal C}}

\newcommand{\NN}{{\mathcal N}}
\newcommand{\LL}{{\mathcal L}}

\newcommand{\bfor}{{\bf for }} 
\newcommand{\bif}{{\bf if }}
\newcommand{\bthen}{{\bf then}:}
\newcommand{\bthenn}{{\bf then }}
\newcommand{\belse}{{\bf else}:}

\newcommand{\bdo}{{\bf do}:}

\newcommand{\bwhile}{{\bf while }}
\newcommand{\bto}{{\bf to }}

\newcommand{\breturn}{{\bf return}}

\newcommand{\be}[1]{\begin{enumerate} [#1]}
\newcommand{\ee}{\end{enumerate}}

\newcounter{cnt1}
\newcounter{cnt2}
\newcounter{cnt3}
\newcounter{cnt4}

\newcommand{\bnum}
{
\begin{list}{\arabic{cnt1})}
{
\usecounter{cnt1}
\leftmargin 5mm
\setlength{\leftmargin}{\leftmargin}
\topsep 2pt
\parsep 1pt
\itemsep 1pt}
}
\newcommand{\enum}{\end{list}}

\newcommand{\broman}
{
\begin{list}{\roman{cnt2})}
{
\usecounter{cnt2}
\leftmargin 2mm
\setlength{\leftmargin}{\leftmargin}
\topsep 2pt
\parsep 1pt
\itemsep 1pt}
}
\newcommand{\eroman}{\end{list}}

\newcommand{\bRoman}
{
\begin{list}{\Roman{cnt3})}
{
\usecounter{cnt3}
\leftmargin 5mm
\setlength{\leftmargin}{\leftmargin}
\topsep 2pt
\parsep 1pt
\itemsep 1pt}
}
\newcommand{\eRoman}{\end{list}}

\newcommand{\balph}
{
\begin{list}{\alph{cnt4})}
{
\usecounter{cnt4}
\leftmargin 3mm
\setlength{\leftmargin}{\leftmargin}
\topsep 2pt
\parsep 1pt
\itemsep 1pt}
}
\newcommand{\ealph}{\end{list}}

\newcommand{\bAlph}
{
\begin{list}{\Alph{cnt1})}
{
\usecounter{cnt1}
\leftmargin 5mm
\setlength{\leftmargin}{\leftmargin}
\topsep 2pt
\parsep 1pt
\itemsep 1pt}
}
\newcommand{\eAlph}{\end{list}}

\newcommand{\bdot}
{
\begin{list}{$\cdot$}
{
\leftmargin  3mm
\setlength{\leftmargin}{\leftmargin}
\topsep 3pt
\parsep 1pt
\itemsep 2pt}
}
\newcommand{\edot}{\end{list}}

\newcommand{\bdash}
{
\begin{list}{-}
{
\leftmargin 3mm
\setlength{\leftmargin}{\leftmargin}
\topsep 2pt
\parsep 1pt
\itemsep 1pt}
}
\newcommand{\edash}{\end{list}}

\newcommand{\bnull}
{
\begin{list}{}
{
\leftmargin 3mm
\setlength{\leftmargin}{\leftmargin}
\topsep 0pt
\parsep 1pt
\itemsep 1pt}
}
\newcommand{\enull}{\end{list}}

\begin{document}


\author{Junichiro Fukuyama}
\address{Department of Computer Science and Engineering\\
The Pennsylvania State University\\
PA 16802, USA}
\curraddr{}
\email{jxf140@psu.edu}
\thanks{}


\date{}

\dedicatory{}


\title{Sunflower Bound with a Sub-Logarithmic Base}

\begin{abstract} 
We show that a family $\FF$ of sets each of cardinality $m \in \Z_{>2}$ includes a $k$-sunflower if $
|\FF| \ge \lp \frac{c k^2 \ln m}{\ln \ln m} \rp^m
$ for some constant $c>0$, where {\em $k$-sunflower} means a family of $k$ different sets with a common pairwise intersection. The base of the exponential lower bound is sub-logarithmic for each $k$ updating the current best-known result. 
\end{abstract}

\maketitle

\section{Motivation, Basic Terminology and Related Facts} 
In this paper, we verify\footnote{ 
Extensive additional information on the proof can be found in \cite{blog}. 
} the following statement.
\begin{theorem} \label{KSF}
There exists $c \in \R_{>0}$ such that for every $k, m \in \Z_{>2}$, a family $\FF$ of sets each of cardinality $m$ includes a $k$-sunflower if
$
|\FF| \ge \lp \frac{c k^2 \ln m}{\ln \ln m} \rp^m
$. 
\qed
\end{theorem}

\noindent 
Here a {\em $k$-sunflower} means a family of $k$ different sets with a common pairwise intersection called the {\em core}, and $\ln$ stands for natural logarithm.

This updates the current best-known result $(c \ln m)^m$ for the case $k=3$ \cite{ALWZ20}, with an asymptotically smaller base $O \lp \frac{\ln m}{\ln \ln m} \rp$ of the exponential lower bound. 
It could be a step forward to proof of the {\em sunflower conjecture} \cite{E81} whose status has not been determined since the {\em sunflower lemma} \cite{blog} was proven in 1960.

\medskip

The rest of the section describes our basic terminology similarly to \cite{blog, KDC}. 
Let $X$ be the universal set, $n$ be its cardinality, and $\ep \in (0, 1)$ a sufficiently small positive number depending on no other variables. Denote 
\(
&& 
[b] = \lbr 1, b \rbr \cap \Z, \quad \textrm{for~} b  \in \R_{>0}, 
\\ && 
{X'\choose m} = \lb U ~:~ U \subset X',~|U| =m  \rb, \quad \textrm{for~} X' \subset X 
\textrm{~and~} m \in [n], 
\\ && 
\FF^r = \underbrace{
\FF \times \FF \times \cdots \times \FF
}_r, \quad \textrm{for~} \FF \subset 2^X 
\textrm{~and~} r \in \Z_{\ge 0}, 
\\ \textrm{and} && 
\FF[S]  = \lb U ~:~ U \in \FF,~ S \subset U  \rb, \quad \textrm{for~} 
S \in 2^X. 
\)

A set means a subset of $X$, and is an {\em $m$-set} if it is in ${X \choose m}$. 
A family $\FF \subset {X \choose m}$ is said to satisfy the {\em $\Gamma(b)$-condition} if $|\FF[S]| < b^{-|S|} |\FF|$ for every nonempty set $S$.

Denote set/family subtraction by $-$. Use $\leftarrow$ to express substitution into a variable. For simplicity, a real interval may denote the integral interval of the same range; $e.g.$, use $\lbr 0, t \rp$ instead of $\lbr 0, t \rp \cap \Z$ if it is clear by context. 
Obvious floor/ceiling functions may be omitted.


\medskip

We have the double inequality 
\[
\lp \frac{x}{y} \rp^y \le  {x \choose y} < \lp \frac{e x}{y} \rp^y, \quad 
\textrm{for}~ x\in \Z_{>0} ~\textrm{and}~ y \in [x], 
\]
called the {\em standard estimate} of the binomial coefficient ${x \choose y}$, where $e=2.71...$ is the natural logarithm base. It is straightforward to check\footnote{
The lower bound holds since 
$
{x \choose y} = \prod_{i=0}^{y-1} \frac{x - i}{y-i}  \ge \lp \frac{x}{y} \rp^y 
$. 
The upper limit is due to $
\sqrt{2 \pi y} \lp \frac{y}{e} \rp^y
\exp \lp \frac{1}{12y+1} \rp
< y! <
\sqrt{2 \pi y} \lp \frac{y}{e} \rp^y
\exp \lp \frac{1}{12y} \rp$ for all $y \in \Z_{>0}$ shown in \cite{R55}, which implies $y! > \lp \frac{y}{e} \rp^y$. So, 
${x \choose y} \le \frac{x^y}{y!} < \lp \frac{ex}{y} \rp^y$. 
} its truth by Stirling's approximation.

\section{Proof of \refth{KSF}}  
With an independent constant $\ep \in (0, 1)$ given as above and integers $k, m \in \Z_{>2}$, denote 
\(
&& 
c = \exp \lp \ep^{-1} \rp, \quad 
i, j, r \in [0, m], \quad 
b_\dag =k^2 \exp \lp c \rp, \quad  b_* = k^2 m^4, \quad 
\\ && \nonumber 
\alpha =\frac{\ln m}{\ln \ln m}, \quad 
\beta = \lc \frac{m}{c \alpha} \rc, \quad 
\gamma = b_\dag^{-\beta} \exp \lp -c \beta \rp, \quad 
b =  k^2 \alpha  \exp \lp c^2 \rp. \quad 
\) 
Assume WLOG that $\FF$ given by \refth{KSF} satisfies the $\Gamma(b)$-condition, which implies $|\FF|>b^m$.  Also $m > k^{-1} b$, otherwise the claim is clear by the sunflower lemma. 
We detect a $k$-sunflower in such an $\FF$ in the following four steps to prove the theorem. 
\be{{\hspace*{2mm}Step} 1:}
\item construct some subfamilies of $\FF^2 \times 2^X$. 
\item run an algortihm to find potential sunflower cores $C$. 
\item select the final core $C$. 
\item detect a $k$-sunflower in $\FF[C]$. 
\ee

\medskip 

{\bf Step 1:~}{\em construct some subfamilies of $\FF^2 \times 2^X$.} Let 
\(
&& 
\PP_j = \lb (T, U)~:~  (T, U)\in \FF^2,~ |T \cap U| =j \rb, 
\quad \textrm{for $j \in [0, m]$}, 
\\ && 
\NN = \bigcup_{j=0}^\beta \PP_j, \quad 
\EE = \bigcup_{j = \beta+1}^m \PP_j, \quad 
\CC = \bigcup_{j=0}^\beta {X \choose j}, \quad 
\\ && 
\NN^* = \lb (T, U, C) ~:~(T, U) \in \NN,~ C \in 2^{T \cap U} \rb, 
\\ && 
\EE^* = \lb (T, U, C) ~:~(T, U) \in \EE,~ C \in \CC \cap 2^{T \cap U} \rb. 
\) 
Extend the $[\cdot]$-notation to denote 
\(
&&
\PP[C] = \PP \cap  \FF[C]^2, \quad \textrm{for any $\PP \subset \FF^2 $ and $C \in 2^X$,}
\\ \textrm{and} && 
\NN'[C] = \NN' \cap \lp \FF[C]^2 \times \lb C \rb \rp,
\quad \textrm{for $\NN' \subset \NN^* \cup \EE^*$.} 
\)

A pair $(T, U) \in \EE[C]$ is said to be an {\em error} as opposed to a $(T, U) \in \NN[C]$ being a non-error set pair. Their incident triples $(T, U, C) \in \EE^*[C] \cup \NN^*[C]$ are the {\em shift-downs of $(T, U)$ to $C$}. 

\medskip

Confirm the three remarks.

\be{A)} 
\item $\NN[C] \cup \EE[C] = \FF[C]^2$ for $C \in \CC$, such that a $(T, U) \in \FF[C]^2$ is an error if and only if $|T \cap U| >\beta$.


\item $|\EE| \le |\EE^*|< (c^\beta b_\dag)^{-1} \gamma |\FF|^2$ meaning $\left| \NN^* \right| \ge |\NN| > \lbr 1- (c^\beta b_\dag)^{-1} \gamma \rbr |\FF|^2$ by A). For, 
\[ 
|\PP_j| < {m \choose j} b^{-j} |\FF|^2 < \lp \frac{b j}{4 m} \rp^{-j} |\FF|^2 < (e^{4c} k^2)^{-j} |\FF|^2, 
\] 
for all $j \in (\beta, m]$, by the $\Gamma(b)$-condition of $\FF$ and the standard estimate of ${m \choose j}$. As each $(T, U) \in \PP_j$ generates less than $2^j$ shift-downs $(T, U, C)$, it follows that 
\(
\frac{|\EE^*|}{|\FF|^2} &<& \sum_{j=\beta+1}^{m} \lp 2^{-1} e^{4c} k^2 \rp^{-j} 
< (e^{3c} k^2)^{-\beta-1} < \frac{\gamma}{c^\beta b_\dag}, 
\)
justifying the inequality. 

\item Let the condition $\psi_1(C)$ for $C \in \CC$ be true if and only if $|\EE[C]| < b_\dag^{-1} |\NN[C]|$. 
Then 
\(
\sum_{C \in \CC,~\psi_1(C)} |\NN[C]|  &=& \sum_{C \in \CC,~\psi_1(C)} ~|\NN^*[C]| 
\ge |\NN^*| - b_\dag |\EE^*|
\\ &>& |\NN^*|- c^{-\beta}\gamma |\FF|^2 
> \lp 1- 2 c^{-\beta}\gamma \rp |\NN^*|, 
\)
by B), $|\NN[C]| = |\NN^*[C]|$ and $|\EE[C]| = |\EE^*[C]|$. 
\ee 


\medskip

{\bf Step 2:~}{\em run an algortihm to find potential sunflower cores $C$.}  
Let 
\(
&& 
\CC_j = \lb C~:~ C \in {X \choose j},~ 
|\FF[C]| \ge b_*^{-j} |\FF| 
\rb, 
\\&& 
\NN^*_j = \bigcup_{i \in [j, \beta],~ C \in \CC_i}  ~\NN^*[C],  
\qquad \textrm{for~} j \in [0, m]. 
\)
Given a subfamily $\NN' \subset \NN^*$, let $r$ be the maximum integer in $[0, \beta]$ such that 
\beeq{eqStep2}
|\NN' \cap \NN^*_r| \ge b_*^{-2r} |\NN'|. 
\eeq
The family 
\[
\MM = \NN' \cap \NN^*_r 
\]
is said to be the {\em $\Gamma(b_*)$-subfamily of $\NN'$ of rank $r$}.

\medskip 

Construct such an $\MM$ by the following algorithm \ltt{R}: 

\begin{small}
\medskip 
\be{\qquad 1.} 
\item $\NN' \leftarrow 
\NN^* - \lb (T, U, C') ~: \exists (T, U, C) \in \NN^*~\textrm{such that}~ \neg \psi_1(C)  ~\wedge~ C' \in 2^{T \cap U} \rb$; 
\item \bwhile $True$ \bdo
	\be{{2}-1.}
	\item $\MM \leftarrow$ the $\Gamma(b_*)$-subfamily of $\NN'$ of rank $r$; 
	\item \bif $|\MM| \ge b_\dag^{- 3\beta} |\NN^*|$ \bthenn \breturn$(\MM)$ \belse 
		\broman 
		\item $\NN' \leftarrow \NN' - \lb (T, U, C') ~:
		\exists (T, U, C) \in \MM ~\textrm{such that}~C' \in 2^{T \cap U} \rb$; 
		\item \bfor each $C \in \CC_r$ \bdo 
			\bnull
			\item \bif $|\NN'[C]| \le \lp 1- b_\dag^{-1} \rp |\NN^*[C]|$ \bthen
				\bnull
				\item $\NN' \leftarrow \NN' - \lb (T, U, C') ~:~ 
				\exists (T, U, C) \in \NN'[C] ~\textrm{such that}~C' \in 2^C \rb$; 
				\enull 
			\enull 
		\eroman 
	\ee 
\ee 
\medskip 
\end{small}

\noindent 
Call $(T, U, C)$ and $(T, U, C')$ in Step 1, 2-2-i) or 2-2-ii) {\em siblings} of each other.

\medskip 

Remarks. 

\be{A)} \setcounter{enumi}{3}
\item Note the four recursive properties of \ltt{R}.  
\broman 
\item The $\Gamma$-subfamilies $\MM$ are found in the order of decreasing ranks. For, \ltt{R} finds each $\MM = \NN' \cap \NN_r^*$ by \refeq{eqStep2} to delete from the current $\NN'$ by Step 2-2-i). There, $\MM$ includes every $(T, U, C) \in \NN'$ such that $|C|\ge r$ and $|\FF[C]| \ge b_*^{-|C|} |\FF|$. As those $(T, U, C)$ are eliminated from $\NN'$, no other $\MM$ of rank $r$ or greater would be detected subsequently. 

\item The remaining $\NN'$ satisfies the two properties after Step 2-2-ii) is executed each time.  
\be{a)}
\item If a shift-down $(T, U, C) \in \NN^*$ with $|C|=r$  is not in the current $\NN'$, its sibling $(T, U, C')$ with $C' \subset C$ is not either. 
\item If $(T, U, C) \in \NN'$ and $|C|=r$, all $(T, U, D)$ with $D \in 2^{T \cap U}[C]$ are in $\NN'$. 
\ee 
b) is true since if $(T, U, D) \not \in \NN'$ for $C \subsetneq D$ due to Step 1, 2-2-i) or ii), the step would have elminated $(T, U, C)$ as well.

\item The ratio $\delta_i = \frac{|\NN^*| - |\NN'|}{|\NN^*|}$ at Step 2-2-ii) after the detection of the $i^{th}$ subfamily $\MM$ satisfies the recurrent relation 
\beeq{eqRecur}
\delta_0 < \gamma, \quad \delta_{i+1} < 2 b_\dag \lp \delta_i  + \frac{2^\beta}{b_\dag^{3\beta}}  \rp. 
\eeq
Its truth is seen as follows. By Remark C), $|\NN^*- \NN'| < \gamma |\NN^*|$ at Step 1. Step 2-2-i) clearly eliminates less than 
$2^\beta b_\dag^{-3\beta}$ of the current $\NN'$ to continue to the next round. To see that 2-2-ii) eliminates extra $(1+\ep) b_\dag |\NN^* - \NN'|$ shift-downs or less, consider each $C$ found by the step such that $|\NN'[C]| \le \lp 1- b_\dag^{-1} \rp |\NN^*[C]|$, shift-downs $(T, U, C) \in \NN^*[C] - \NN'[C]$, and their siblings 
$(T, U, C') \in \NN^*[C'] - \NN'[C']$ at $C' \in 2^C- \lb C \rb$. 
We see that: 
\bdash 
\item $|\NN'[C]| < (1+ \ep) b_\dag |\NN^*[C] - \NN'[C]|$. 
\item For each $(T, U, C)$ and $C'$, the sibling $(T, U, C')$ is included in $\NN^*[C']- \NN'[C']$ by D)-ii)-a), meaning $|\NN^*[C] - \NN'[C]| \le |\NN^*[C'] - \NN'[C']|$. 
\item There are no more than 
$ 
|\NN'[C]| < (1+ \ep) b_\dag |\NN^*[C] - \NN'[C]| \le (1+ \ep) b_\dag |\NN^*[C'] - \NN'[C']| 
$
deleted siblings $(T, U, C') \in \NN'[C']$ for each $C'$. 
\edash 
As $\NN^*[C]$ and $\NN^*[C']$ are mutually disjoint, these justify the desired upper-bound $(1+\ep)b_\dag |\NN^* - \NN'|$ proving \refeq{eqRecur}.

\item It means 
\[
\delta_i + a \le (2 b_\dag)^i \lp \delta_0 +  a \rp, 
\eqwhere a= \frac{2^\beta}{b_\dag^{3\beta} \lp 1 - \frac{1}{2b_\dag}\rp}. 
\]
Since $\delta_0 < \gamma= (e^c b_\dag)^{-\beta}$, we have $\delta_i < c^{-\beta}$ for every $i$. Thus, $|\NN'| >(1- c^{-\beta}) |\NN^*|$ throughout the execution of \ltt{R}. 
\eroman

\item Initiated by its first step holding for $|\NN'| > \lp 1- \gamma \rp \left| \NN^* \right|$, \ltt{R} continues on Loop 2 satisfying D)-i) to iv). 
The process does not halt with an exception since there always exists the $\Gamma(b_*)$-subfamily of the current $\NN'$. After its termination, it outputs $\MM \subset \NN^*$ that is defined for any $\FF$ as given above. Note here that Step 2-2-ii) eliminates the same family of shift-downs regardless of the order of detecting $C \in \CC_r$.

\item The following holds for the final $\MM$ of rank $r$. 
\broman 
\item 
\[
|\MM| \ge b_\dag^{-3\beta} |\NN^*| > ( e^c k^2)^{-3 \beta} (1-\gamma) |\FF|^2
> (e^c k^2 \ln m)^m. 
\]

\item $|\MM \cap \NN^*_j| < b_*^{-2(j-r)} |\MM|$ for each $j \in (r, \beta]$. For, $|\NN' \cap \NN^*_j|< b_*^{- 2j} |\NN'|$ right before $\MM$ was detected by Step 2-1 of \ltt{R}. It means 
\[
|\MM \cap \NN^*_j| \le 
|\NN' \cap \NN^*_j|< b_*^{- 2 j} |\NN'| \le b_*^{- 2(j-r)} |\MM|. 
\]

\item  $\MM[C]$ for $C \in \CC_r$ are mutually disjoint subfamilies of $\MM$ such that 
\[
\left| \bigcup_{C \in \CC_r} \MM[C] \right| > (1 - b_*^{-1}) |\MM|. 
\]
It is true since $\MM - \bigcup_{C \in \CC_r} \MM[C] \subset \MM \cap \NN_{r+1}^*$ is smaller than $b_*^{-1}$ of $\MM$ by F)-ii).

\item $|\MM[C]| > (1- b_\dag^{-1}) |\NN^*[C]|$ for every $C \in \CC_r$ such that $\MM[C] \ne \emptyset$, by the elimination by Step 2-2-ii).

\item For each $C \in \CC_r$ such that $\MM[C] \ne \emptyset$, 
\(
|\FF[C]| &\ge&b_*^{-r} |\FF| > (k^2 m^4)^{- \beta} b^m
> (k^2 m^4)^{- \lc \frac{m \ln \ln m}{c \ln m} \rc} (e^{2c} k^2 \alpha)^m 
\\ &>& \lp e^c k \sqrt{\ln m} \rp^m. 
\)
Also it satisfies $\psi_1(C)$ due to the first step of \ltt{R}. 

\item Let $r' < r$ be the largest rank of the $\Gamma(b_*)$-subfamily detected by \ltt{R}, or $r' = \beta-1$ if such a rank does not exist. By D)-i) and D)-ii): 
\bdash
\item $\MM$ includes no $(T, U, C)$ with $|C| \ge r'$. 
\item For any $(T, U) \in \NN$ and two $C, D \in \bigcup_{i=r}^{\beta} \CC_i$ such that $C \subset D \subset T \cap U$, 
\[
(T, U, C) \in \MM, \quad \Rightarrow \quad (T, U, D) \in \MM. 
\]
In particular, if $r'=r+1$ and $(T, U, C) \in \MM$ for an $r$-set $C$, it has to satisfy $|T \cap U| = r$. 
\item $(T, U, D) \in \MM$ does not necesssarily mean $(T, U, C) \in \MM$, since $(T, U, C)$ could have been deleted when another $(T, U, D') \in \NN^*$ with $C \subset D'$ was. 
\edash 
\eroman 
\ee

\medskip 

{\bf Step 3:~}{\em select the final core $C$.} 
Prove a lemma on the $\MM$ and $r$ with the remarks. 

\begin{lemma} \label{lmSmallness}
\[
\sum_{C \in \CC_r,~ \neg \psi_2(C)}~ \left| \MM[C] \right| 
< \frac{m}{b_*} | \MM |, 
\]
where the condition $\psi_2(C)$ is true if and only if 
\(
&& 
\sum_{D \in \CC_{r+1}[C]}~\left| \MM_{C, D} \right| <\frac{m}{b_*} |\MM[C]|, 
\\ \textrm{where} &&
\MM_{C, D} = \lb (T, U, C) ~:~ (T, U, C) \in \MM[C],
~ 
D \in {T \cap U \choose r+1}[C] \rb. 
\) 
\end{lemma}
\begin{proof}
Suppose to the contrary it were false. Consider all $\MM[C]$ such that $C \in \CC_r$ $\wedge$ $\neg \psi_2(C)$, and $D \in \CC_{r+1}[C]$. 
There would be at least $m^2 b_*^{-2} |\MM|$ such shift-downs $(T, U, C) \in \MM_{C, D} \subset \MM[C]$ in total, counted at all $D$ with multiplicity. So, 
\beeq{eq2Small}
\sum_{C \in \CC_r,~D \in \CC_{r+1}[C]}~\left| \MM_{C, D} \right| \ge \frac{m^2}{b_*^2} |\MM|. 
\eeq

On the other hand, F)-ii) requires $|\MM \cap \NN^*_j| < b_*^{- 2(j-r)} |\MM|$ for $j > r$, 
where the two sides count shift-downs $(T, U, D') \in \MM[D']$ to $D' \in \CC_j$  and/or $D'$ in other $\CC_{j'}, j' \ne j$. By D)-ii)-b) and F)-vi), if $\MM$ includes a shift-down $(T, U, C)$ such that $|T \cap U| > r$, its sibling $(T, U, D')$ for some $j \in (r, |T \cap U|]$ and $D' \in \CC[C] \cap {T \cap U \choose j}$ is included in $\MM$ as well.

Such a $(T, U, D')$ produces no more than 
\[
{|D'| \choose r+1} {r+1 \choose r} = (r+1){j \choose r+1} 
= \frac{j!}{r! (j-r-1)!} = (j-r) {j \choose j-r} 
\]
quintuples $(T, U, C, D, D')$ such that $(T, U, C) \in \MM_{C, D}$ and 
$D \in {D' \choose r+1}[C]$. Those $(T, U, C, D, D')$ for any $j$ can be mapped onto all the $(T, U, C) \in \MM_{C, D}$ with their duplicates at distinct $D \in \CC_{r+1}[C]$ considered separately. 

Therefore, 
\(
\sum_{C \in \CC_r,~D \in \CC_{r+1}[C]}~\left| \MM_{C, D} \right|
&\le& 
\sum_{j=r+1}^\beta (j-r) {j \choose j- r} b_*^{-2j+2r} |\MM|
\\ &<&  \sum_{j=r+1}^\beta \lp \frac{b^2_*}{j} \rp^{-j+r} |\MM|
< \frac{m}{b^2_*} |\MM|. 
\) 
It contradicts \refeq{eq2Small} proving the desired inequality. 
\end{proof}

\medskip

Due to the remarks and the lemma, there exists $C \in \CC_r$ such that $\MM[C] \ne \emptyset$ and $\psi_2(C)$. Fix such a $C$ to confirm the four conditions: 
\be{{G)}-i)} 
\item  $|\EE[C]| < b_\dag^{-1} |\NN[C]|$ by $\psi_1(C)$. 
\item $|\MM_C| > (1- b_\dag^{-1}) |\NN[C]|$ 
where $\MM_C = \lb (T, U)~:~(T, U, C) \in \MM[C] \rb$, by F)-iv) and $|\NN[C]|= |\NN^*[C]|$ from C). 
\item $\sum_{D \in \CC_{r+1}[C]}~\left| \MM_C[D] \right| < m b_*^{-1} |\MM_C|$ by $\psi_2(C)$. 
\item $|\FF[C]| \ge b_*^{-r} |\FF| > \lp e^c k \sqrt{\ln m} \rp^m$ by F)-v). 
\ee 

\medskip 

It will be the core of our $k$-sunflower in the last step. We particularly note with $\bigcup_{D \in {X \choose \beta+1}[C]} \MM_C[D] =\emptyset$ that the four conditions hold even if $r= \beta$. See other properties of $C$. 

\medskip 

\be{A)} \setcounter{enumi}{7}
\item By A), G)-i) and G)-ii), 
\[
|\FF[C]^2 - \MM_C| = |\EE[C]| + \left| \NN[C] - \MM_C \right| 
< \frac{2}{b_\dag}|\FF[C]|^2. 
\]

\item The family 
\[
\LL = \bigcup_{D \in \DD}~ \MM[D], 
\eqwhere 
\DD = \lb D ~:~ D \in 2^X[C] - \lb C \rb,~|\FF[D]| < b_*^{-|D|} |\FF| \rb, 
\]
satisfies 
\(
|\LL| &<& |\FF[C]|  \sum_{j = r+1}^m~ {m-r \choose j-r} b_*^{-j}  |\FF|  
\le |\FF[C]|^2 \sum_{j = r+1}^m~ {m-r \choose j-r} b_*^{-j+r} 
\\ &\le& |\FF[C]|^2 \sum_{j = r+1}^m~ \lbr \frac{(j-r) b_*}{3(m-r)} \rbr^{-j+r} 
< \frac{4(m-r)}{b_*} |\FF[C]|^2. 
\)

\item 
By the above two and G)-iii), 
\(
\left| \bigcup_{D \in 2^X[C] - \lb C \rb} \MM_C[D] \right| 
&\le& |\LL|  + \sum_{D \in \CC_{r+1}[C]}~\left| \MM_C[D] \right|
\\ &<& \frac{5m |\FF[C]|^2}{b_*} <\frac{6m |\MM_C|}{b_*}. 
\)
\ee

\medskip 

{\bf Step 4:~}{\em detect a $k$-sunflower in $\FF[C]$.} 
We have a statement from the arguments so far. 

\begin{lemma}
There exist  $[1- (c k)^{-1}] |\FF[C]| > (e^c k)^m$ sets $T \in \FF[C]$ each satisfying the folloing condition.  
\bnull
\item $\psi_4(T)$: $T \cap U = C$ for more than $[1- (c k)^{-1}] |\FF[C]|$ sets $U \in \FF[C]$. 
\enull 
\end{lemma}
\begin{proof}
By G), H), J) and $m > k^{-1} b$: 
\be{\hspace*{3mm} -}
\item $|\FF[C]| > \lp e^c k \sqrt{\ln m}\rp^m$. 

\item $|\FF[C]^2 - \MM_C| < (c^2 k)^{-2} |\FF[C]|^2$. 

\item $\left| \bigcup_{D \in \CC[C] - \lb C  \rb} \MM_C[D] \right| 
< (c^2 k)^{-2} |\FF[C]|^2$. 
\ee
The three mean that there are more than $[1 - 2 (c^2k)^{-2}] |\FF[C]|^2> 2^{-1}\lp e^{2c} k^2 \ln m \rp^m $ pairs $(T, U) \in \FF[C]^2$ such that $T \cap U = C$, leading to the desired claim. 
\end{proof}

We now construct a $k$-sunflower $\TT = \lb T_1, T_2, \ldots, T_ k \rb \subset \FF[C]$ by the following algorithm \ltt{S}: 

\begin{small}
\medskip 

\be{\hspace*{3mm} 1.}
\item 
$\GG \leftarrow \FF[C]$; \hspace*{1mm}
$\TT \leftarrow \emptyset$; \hspace*{1mm}
\item \bfor $j=1$ \bto $k$ \bdo
	\be{{2}-1.}
	\item find $T_j \in \GG$ such that $\psi_4(T_j)$; 
	\item $\TT \leftarrow \TT \cup \lb T_j \rb$; 
	\item delete all $U \in \GG$ such that $T_j \cap U \ne C$ from $\GG$; 
	\ee 
\ee 

\medskip 

\end{small}

\noindent 
The recursive invariant of \ltt{S} is the two conditions after Step 2-3 is executed for each $j \in [k]$: 
\be{\hspace*{1mm} a)}
\item $|\GG| \ge [1 - j (ck)^{-1} ] |\FF[C]| > (c k)^m$. 
\item $T_i \cap U = C$ for all $i \in [j]$ and $U \in \GG \cup \TT - \lb T_i \rb$. 
\ee

\medskip 

a) is clearly true for every $j$, because Step 2-3 could delete $(ck)^{-1} |\FF[C]|$ or less $U$ due to $\psi_4(T_j)$ of the selected $T_j \in \GG$. This also guarantees the existence of such a $T_j$ because there are $[1 - (ck)^{-1} ] |\FF[C]|$ sets $T \in \FF[C]$ with $\psi_4(T)$ due to the lemma, thus $[1 - (j-1)(ck)^{-1} ] |\FF[C]|$ in the current $\GG$ before Step 2-1.

As Step 2-3 eliminates all $U \in \GG$ such $T_j \cap U \ne C$, the condition b) is true as well. 

\medskip 

Those confirm that the two are true when \ltt{S} terminates with $\TT=\lb T_1, T_2, \ldots, T_k \rb$. By b), it is a $k$-sunflower in $\FF$ with the core $C$. We have proven \refth{KSF}.




\bibliographystyle{amsplain}

\end{document}